\documentclass[11pt]{amsart}

\usepackage[T1]{fontenc}
\usepackage{lmodern}
\usepackage{amsmath,amssymb,mathtools}
\usepackage{enumitem}
\usepackage{microtype}
\usepackage{url}
\usepackage[colorlinks=true,citecolor=blue,linkcolor=blue,urlcolor=blue]{hyperref}

\newtheorem{theorem}{Theorem}[section]
\newtheorem{corollary}[theorem]{Corollary}
\newtheorem{proposition}[theorem]{Proposition}
\newtheorem{lemma}[theorem]{Lemma}
\newtheorem{remark}[theorem]{Remark}
\theoremstyle{definition}
\newtheorem{definition}[theorem]{Definition}

\newcommand{\CA}{\operatorname{CA}}

\newcommand{\Nzero}{\mathbb{N}_0}
\newcommand{\Ftwo}{\mathbb{F}_2}

\title[A reversibility characterization of locally finite groups]{A Reversibility Characterization of Locally Finite Groups by Cellular Automata}

\author{Jiang Yang}
\address{School of Mathematical Sciences, Guangxi Minzu University, Nanning, China}
\email{yangjiangdy@126.com}


\begin{document}

\begin{abstract}
For cellular automata over finite alphabets, bijectivity already implies reversibility.  Over infinite alphabets this implication may fail, and the remaining obstruction in the periodic case was recorded by Ceccherini-Silberstein and Coornaert as Open Problem 2 in \emph{Cellular Automata and Groups}.  We prove an exact group-theoretic characterization.  A group $G$ is locally finite if and only if, over every alphabet, every bijective cellular automaton $A^G\to A^G$ is reversible.  Equivalently, if $G$ is not locally finite, then for every infinite alphabet $A$ there exists a bijective cellular automaton $A^G\to A^G$ whose inverse is not a cellular automaton.  The counterexample is already obtained on a countable alphabet.  Its local rule has a rank track, a direction track and a binary data track; the forward map is triangular along finite directed chains of arbitrary length, so its inverse is defined pointwise but has no uniform finite memory.  As a consequence, Open Problem 2 has an affirmative answer, and the periodicity hypothesis is unnecessary for the negative direction.
\end{abstract}

\maketitle

\section{Introduction}

Cellular automata over groups provide a common language for symbolic dynamics, geometric group theory and finite-memory computation.  Their interaction with group-theoretic properties is already visible in the Garden of Eden theorem for amenable groups; see \cite{CSMS1999} and the systematic treatment in \cite{CAG2}.  If $G$ is a group and $A$ is a set, a cellular automaton $\tau\colon A^G\to A^G$ is a shift-equivariant map whose value at each site is determined by a fixed finite neighbourhood.  When $A$ is finite, the classical compactness argument behind the Curtis-Hedlund-Lyndon theorem implies that every bijective cellular automaton is reversible, i.e. its inverse is again a cellular automaton; see, for instance, \cite[Theorem 1.10.2]{CAG2} and the original symbolic-dynamical source \cite{Hedlund1969}.  This finite-alphabet phenomenon is one of the basic reasons why reversibility is well behaved in classical cellular automata theory.

For infinite alphabets the situation is different.  The inverse of a bijective cellular automaton may be a well-defined shift-equivariant map without having any finite memory.  The monograph \cite{CAG2} gives such examples over $\mathbb Z$ and records a family of open problems asking how this failure interacts with the algebraic structure of the acting group.  In particular, Open Problem 2 in the second edition asks whether, for a periodic group $G$ which is not locally finite and an infinite alphabet $A$, there exists a bijective cellular automaton
\[
        \tau\colon A^G\longrightarrow A^G
\]
which is not invertible in the cellular-automaton sense; see \cite[Open Problems, OP-2]{CAG2}.  The same source notes the two boundary cases known before: the answer is positive when $G$ contains an element of infinite order, while if $G$ is locally finite then every bijective cellular automaton over an arbitrary alphabet is invertible; see \cite[Comments on OP-2]{CAG2} and \cite[Corollary 1.2 and Proposition 4.1]{CC2011Reversibility}.

Thus the genuine gap was the Burnside-type situation: finitely generated infinite torsion subgroups, such as the first Grigorchuk group \cite{Grigorchuk1980}, contain no infinite cyclic direction along which one can simply shift information to infinity.  The point of this note is that an infinite-order element is not needed.  The right replacement is a variable-labelled ray in a Cayley graph, together with a rank track which makes all computations finite in each individual configuration but with no uniform bound on the length needed by the inverse.

Our main result is the following exact characterization.

\begin{theorem}[Main theorem]\label{thm:main}
For a group $G$, the following conditions are equivalent.
\begin{enumerate}[label=\textup{(\roman*)}]
    \item $G$ is locally finite.
    \item For every nonempty alphabet $A$, every bijective cellular automaton $A^G\to A^G$ is reversible.
    \item For every infinite alphabet $A$, every bijective cellular automaton $A^G\to A^G$ is reversible.
    \item For some infinite alphabet $A$, every bijective cellular automaton $A^G\to A^G$ is reversible.
\end{enumerate}
Equivalently, if $G$ is not locally finite, then for every infinite alphabet $A$ there exists a bijective cellular automaton $A^G\to A^G$ whose inverse is not a cellular automaton.
\end{theorem}

The proof is constructive.  Starting with a finite set which generates an infinite subgroup of $G$, we choose a finite symmetric set $S$ and use the countable alphabet
\[
        B=\Nzero\times (S\sqcup\{\bot\})\times \Ftwo .
\]
The three components are interpreted as rank, direction and data.  At a site $g$, the automaton either leaves the data bit fixed or adds to it the data bit at a directed neighbour $gd(g)$, but only when the rank strictly drops by one.  The forward rule has memory $\{1_G\}\cup S$.  For each fixed control configuration, the data map is $I+P$, where $P$ is pointwise nilpotent: from a site of rank $n$, all paths have length at most $n$.  Hence $I+P$ is bijective with inverse $I+P+\cdots+P^n$ at that site.  However, along finite directed chains of arbitrarily large length, the inverse at the first site reads the parity of the whole chain.  No finite memory can do this uniformly.

This result strengthens the non-periodic examples of \cite{CC2011Reversibility}: instead of relying on an infinite-order element, it characterizes exactly those groups for which bijective cellular automata over infinite alphabets can fail to be reversible.  It is also parallel in spirit to the linear local-finiteness characterization of \cite{CC2011LinearLF}, but it concerns arbitrary alphabets and bijective, not merely linear, cellular automata.  It shows that Open Problem 2 has an affirmative answer in all periodic non-locally finite cases, and in fact in every non-locally finite group.

\section{Cellular automata and reversible cellular automata}

We fix the convention used throughout the paper.  An \emph{alphabet} is a nonempty set.  If $G$ is a group and $A$ is an alphabet, then $A^G$ is the set of configurations $x\colon G\to A$.  The group $G$ acts on $A^G$ by the left shift
\[
        (g x)(h)=x(g^{-1}h),
        \qquad g,h\in G.
\]

\begin{definition}
A map $\tau\colon A^G\to A^G$ is a \emph{cellular automaton} if there exist a finite subset $M\subset G$ and a map $\mu\colon A^M\to A$ such that
\[
        \tau(x)(g)=\mu\bigl((g^{-1}x)|_M\bigr)
\]
for all $x\in A^G$ and $g\in G$.  Equivalently, $\tau(x)(g)$ depends only on the restriction of $x$ to $gM$.  Such a set $M$ is called a \emph{memory set} for $\tau$.
\end{definition}

A bijective cellular automaton $\tau\colon A^G\to A^G$ is called \emph{reversible} if its inverse map $\tau^{-1}\colon A^G\to A^G$ is also a cellular automaton.  We write $\CA(G;A)$ for the monoid of cellular automata $A^G\to A^G$.

A group $G$ is \emph{locally finite} if every finitely generated subgroup of $G$ is finite.  The following graph-theoretic fact will be used in the construction.

\begin{lemma}[Rays in infinite finitely generated groups]\label{lem:ray}
Let $H$ be an infinite group generated by a finite symmetric set $S=S^{-1}$ with $1_H\notin S$.  Then there exist pairwise distinct elements
\[
        h_0=1_H,h_1,h_2,\ldots
\]
and elements $s_1,s_2,\ldots\in S$ such that
\[
        h_i=h_{i-1}s_i
        \qquad (i\ge 1).
\]
\end{lemma}

\begin{proof}
The Cayley graph $\Gamma=\operatorname{Cay}(H,S)$ is connected, infinite and locally finite.  For each $n$ choose a vertex outside the ball of radius $n$ about $1_H$ and take a simple path from $1_H$ to that vertex.  Since only finitely many edges leave each finite ball, a standard diagonal extraction, equivalently König's infinity lemma, gives an infinite simple path starting at $1_H$.  Reading the edge labels along this path gives the required sequence.
\end{proof}

\section{The locally finite case}

We first record the easy half of the characterization.  This result is known, and appears in the comments to OP-2 in \cite{CAG2}; we include the proof because it is short and fixes the coset convention used later.

\begin{proposition}\label{prop:lf}
If $G$ is locally finite, then for every alphabet $A$ every bijective cellular automaton $\tau\colon A^G\to A^G$ is reversible.
\end{proposition}

\begin{proof}
Let $M$ be a finite memory set for $\tau$, and put $H=\langle M\rangle$.  Since $G$ is locally finite, $H$ is finite.  For every $g\in G$ we have $gM\subset gH$, so the restriction of $\tau(x)$ to a left coset $gH$ depends only on the restriction of $x$ to the same left coset.

Thus $\tau$ is the product, over the left cosets of $H$ in $G$, of one finite-block map
\[
        \tau_H\colon A^H\longrightarrow A^H .
\]
More explicitly, identifying $A^{gH}$ with $A^H$ by left multiplication by $g^{-1}$, the action of $\tau$ on every coset is the same map $\tau_H$.  If $\tau_H$ were not injective, then changing one coset would make $\tau$ non-injective.  If $\tau_H$ were not surjective, then prescribing an unattainable pattern on one coset would make $\tau$ non-surjective.  Hence the global bijectivity of $\tau$ implies that $\tau_H$ is bijective.

Since $H$ is finite, the inverse map $\tau_H^{-1}\colon A^H\to A^H$ is implemented by a cellular automaton over the finite group $H$ with memory set $H$.  Applying $\tau_H^{-1}$ independently on each left coset $gH$ gives a cellular automaton $\sigma\colon A^G\to A^G$ with memory set $H$.  By construction, $\sigma=\tau^{-1}$.  Thus $\tau$ is reversible.
\end{proof}

\section{A non-reversible bijective automaton over a countable alphabet}

We now prove the converse direction.  Throughout this section assume that $G$ is not locally finite.  Then there is a finite subset $F\subset G$ such that $\langle F\rangle$ is infinite.  Set
\[
        S=(F\cup F^{-1})\setminus\{1_G\}.
\]
After discarding repetitions, $S$ is finite, symmetric and generates an infinite subgroup.  Put
\[
        D=S\sqcup\{\bot\},
        \qquad
        B=\Nzero\times D\times\Ftwo.
\]
For $x\in B^G$ and $g\in G$ write
\[
        x(g)=\bigl(r_x(g),d_x(g),u_x(g)\bigr),
\]
where $r_x(g)\in\Nzero$, $d_x(g)\in D$, and $u_x(g)\in\Ftwo$.  We call $r_x$ and $d_x$ the control tracks and $u_x$ the data track.

A site $g\in G$ is called \emph{active for $x$} if
\[
        d_x(g)\in S,
        \qquad r_x(g)>0,
        \qquad r_x\bigl(gd_x(g)\bigr)=r_x(g)-1.
\]
Define $T\colon B^G\to B^G$ by
\[
T(x)(g)=
\begin{cases}
\bigl(r_x(g),d_x(g),u_x(g)+u_x(gd_x(g))\bigr),
   & \text{if }g\text{ is active for }x,\\[1mm]
\bigl(r_x(g),d_x(g),u_x(g)\bigr),
   & \text{otherwise.}
\end{cases}
\tag{4.1}\label{eq:Tdef}
\]
The addition is in $\Ftwo$.

\begin{lemma}\label{lem:TCA}
The map $T$ is a cellular automaton with memory set $\{1_G\}\cup S$.
\end{lemma}

\begin{proof}
To compute $T(x)(g)$ one needs $x(g)$ and, if $d_x(g)=s\in S$, the value $x(gs)$.  Thus all required coordinates lie in $g(\{1_G\}\cup S)$.  The local rule is the evident rule induced by \eqref{eq:Tdef}; hence $T$ is a cellular automaton.
\end{proof}

The key point is that $T$ is globally bijective on the full shift $B^G$, without imposing any admissibility condition on the control tracks.

\begin{lemma}\label{lem:Tbijective}
The cellular automaton $T\colon B^G\to B^G$ is bijective.
\end{lemma}

\begin{proof}
Fix a control configuration
\[
        c=(r,d)\in (\Nzero\times D)^G.
\]
On the fibre with this control, the data part of $T$ has the form
\[
        T_c=I+P_c
\]
on $\Ftwo^G$, where
\[
(P_cu)(g)=
\begin{cases}
        u(gd(g)), & \text{if }g\text{ is active for }c,\\
        0,        & \text{otherwise.}
\end{cases}
\]
Here ``active for $c$'' means that the three defining conditions above hold with $r,d$ in place of $r_x,d_x$.

If $g=g_0\to g_1\to\cdots\to g_k$ is an active path for $c$, then the rank drops by one at each step:
\[
        r(g_i)=r(g_0)-i.
\]
Consequently
\[
        (P_c^k u)(g)=0
        \qquad\text{whenever }k>r(g).
\tag{4.2}\label{eq:pointwise-nilpotent}
\]
Thus the formula
\[
        (R_cv)(g)=\sum_{k=0}^{r(g)}(P_c^k v)(g)
\tag{4.3}\label{eq:Rc}
\]
defines a map $R_c\colon\Ftwo^G\to\Ftwo^G$; every coordinate sum is finite.

We claim that $R_c$ is the inverse of $I+P_c$.  First, using \eqref{eq:pointwise-nilpotent}, for every $g\in G$ we get
\[
        \bigl(R_c(I+P_c)v\bigr)(g)
        =\sum_{k=0}^{r(g)}\bigl(P_c^k(I+P_c)v\bigr)(g)
        =v(g)+(P_c^{r(g)+1}v)(g)
        =v(g).
\]
Hence $R_c(I+P_c)=I$.

Conversely, if $g$ is inactive for $c$, then $(P_c w)(g)=0$ for all $w$, and \eqref{eq:Rc} gives $(R_cv)(g)=v(g)$.  Hence $((I+P_c)R_cv)(g)=v(g)$.  If $g$ is active and $h=gd(g)$, then $r(h)=r(g)-1$, and
\[
        (P_c^k v)(h)=(P_c^{k+1}v)(g)
        \qquad(0\le k\le r(h)).
\]
Therefore
\[
\begin{aligned}
((I+P_c)R_cv)(g)
 &= (R_cv)(g)+(R_cv)(h) \\
 &= \sum_{k=0}^{r(g)}(P_c^k v)(g)
    +\sum_{k=0}^{r(g)-1}(P_c^k v)(h) \\
 &= \sum_{k=0}^{r(g)}(P_c^k v)(g)
    +\sum_{k=1}^{r(g)}(P_c^k v)(g) \\
 &=v(g),
\end{aligned}
\]
because the data field is $\Ftwo$.  Thus $(I+P_c)R_c=I$.

Hence each fibre map $T_c$ is bijective, and the control tracks are fixed by $T$.  Therefore $T$ is bijective on $B^G$.
\end{proof}

\begin{lemma}\label{lem:notreversiblecountable}
The inverse map $T^{-1}\colon B^G\to B^G$ is not a cellular automaton.
\end{lemma}

\begin{proof}
By Lemma \ref{lem:ray}, applied to the infinite subgroup generated by $S$, there exist pairwise distinct elements
\[
        h_0=1_G,h_1,h_2,\ldots
\]
and elements $s_1,s_2,\ldots\in S$ such that $h_i=h_{i-1}s_i$ for all $i\ge1$.

Assume, toward a contradiction, that $T^{-1}$ has a finite memory set $\Omega\subset G$.  Choose $N\ge1$ such that
\[
        h_N\notin\Omega.
\]
Define a control configuration $c_N=(r_N,d_N)\in(\Nzero\times D)^G$ as follows:
\[
        r_N(h_i)=N-i,
        \qquad
        d_N(h_i)=s_{i+1}
        \qquad(0\le i<N),
\]
while
\[
        r_N(h_N)=0,
        \qquad
        d_N(h_N)=\bot.
\]
At all other sites $g$ put $r_N(g)=0$ and $d_N(g)=\bot$.  Then the active part of the control graph contains exactly the finite chain
\[
        h_0\longrightarrow h_1\longrightarrow\cdots\longrightarrow h_N .
\tag{4.4}\label{eq:finite-chain}
\]

Now define two output data configurations $v^{(0)},v^{(1)}\in\Ftwo^G$ by
\[
        v^{(0)}\equiv0,
        \qquad
        v^{(1)}(g)=
        \begin{cases}
        1, & g=h_N,\\
        0, & g\ne h_N.
        \end{cases}
\]
Let
\[
        y^{(j)}=(c_N,v^{(j)})\in B^G,
        \qquad j=0,1.
\]
Since the two configurations have the same control tracks and differ only in the data bit at $h_N\notin\Omega$, we have
\[
        y^{(0)}|_{\Omega}=y^{(1)}|_{\Omega}.
\tag{4.5}\label{eq:same-memory}
\]

On the other hand, the inverse formula \eqref{eq:Rc} gives, at the initial vertex $h_0=1_G$ of the chain,
\[
        T^{-1}(c_N,v)(h_0)_{\mathrm{data}}
        =\sum_{i=0}^{N}v(h_i).
\tag{4.6}\label{eq:parity-chain}
\]
Therefore
\[
        T^{-1}(y^{(0)})(1_G)_{\mathrm{data}}=0,
        \qquad
        T^{-1}(y^{(1)})(1_G)_{\mathrm{data}}=1.
\]
This contradicts \eqref{eq:same-memory}, because a cellular automaton with memory set $\Omega$ would have the same value at $1_G$ on any two configurations agreeing on $\Omega$.  Hence $T^{-1}$ is not a cellular automaton.
\end{proof}

Combining the last three lemmas gives the countable-alphabet counterexample.

\begin{proposition}\label{prop:countable-counterexample}
If $G$ is not locally finite, then there exists a countably infinite alphabet $B$ and a bijective cellular automaton $T\colon B^G\to B^G$ whose inverse is not a cellular automaton.
\end{proposition}

\section{Passage to arbitrary infinite alphabets}

The countable alphabet in Proposition \ref{prop:countable-counterexample} is enough to obtain counterexamples over every infinite alphabet.

\begin{lemma}\label{lem:enlargealphabet}
Let $B$ be a countable alphabet and suppose that $T\colon B^G\to B^G$ is a bijective cellular automaton which is not reversible.  Then, for every infinite alphabet $A$, there exists a bijective cellular automaton $\tau_A\colon A^G\to A^G$ which is not reversible.
\end{lemma}

\begin{proof}
Since $A$ is infinite and $B$ is countable, ZFC gives a bijection
\[
        \theta\colon A\longrightarrow A\times B.
\]
The induced one-block conjugacy
\[
        \Theta\colon A^G\longrightarrow (A\times B)^G
\]
is a reversible cellular automaton.  On $(A\times B)^G\cong A^G\times B^G$ define
\[
        \widehat T(a,b)=(a,T(b)).
\]
This is a bijective cellular automaton.  Set
\[
        \tau_A=\Theta^{-1}\circ \widehat T\circ \Theta.
\]
Then $\tau_A$ is a bijective cellular automaton on $A^G$.

It remains to show that $\widehat T$, and hence $\tau_A$, is not reversible.  If $\widehat T^{-1}$ had a finite memory set $\Omega$, then the following would give a finite memory rule for $T^{-1}$.  Fix any $a_0\in A$.  For $b,b'\in B^G$ with $b|_{\Omega}=b'|_{\Omega}$, the configurations $(\bar a_0,b)$ and $(\bar a_0,b')$ in $(A\times B)^G$, where $\bar a_0$ is the constant $A$-configuration, agree on $\Omega$.  Hence their images under $\widehat T^{-1}$ agree at $1_G$.  Projecting to the $B$-track gives
\[
        T^{-1}(b)(1_G)=T^{-1}(b')(1_G).
\]
Thus $T^{-1}$ would have memory set $\Omega$, a contradiction.  Therefore $\widehat T$ and $\tau_A$ are not reversible.
\end{proof}

\section{Proof of the main theorem and consequences}

\begin{proof}[Proof of Theorem \ref{thm:main}]
If $G$ is locally finite, Proposition \ref{prop:lf} gives condition \textup{(ii)}.  Clearly \textup{(ii)} implies \textup{(iii)}, and \textup{(iii)} implies \textup{(iv)}.

Conversely, suppose that $G$ is not locally finite.  Proposition \ref{prop:countable-counterexample} gives a countable infinite alphabet $B$ and a bijective non-reversible cellular automaton on $B^G$.  Lemma \ref{lem:enlargealphabet} then gives such a counterexample over every infinite alphabet $A$.  Hence condition \textup{(iv)} fails.  Therefore \textup{(iv)} implies local finiteness, and all conditions are equivalent.
\end{proof}

The announced answer to OP-2 is now immediate.

\begin{corollary}[Answer to OP-2]\label{cor:OP2}
Let $G$ be a periodic group which is not locally finite, and let $A$ be an infinite set.  Then there exists a bijective cellular automaton
\[
        \tau\colon A^G\longrightarrow A^G
\]
which is not reversible.  In fact, the same conclusion holds for every non-locally finite group $G$, without assuming periodicity.
\end{corollary}

\begin{corollary}\label{cor:grigorchuk}
For the first Grigorchuk group $\mathfrak G$ and every infinite alphabet $A$, there exists a bijective cellular automaton $A^{\mathfrak G}\to A^{\mathfrak G}$ whose inverse is not a cellular automaton.
\end{corollary}

\begin{proof}
The first Grigorchuk group is finitely generated, infinite and periodic \cite{Grigorchuk1980}.  Hence it is not locally finite, and the claim follows from Corollary \ref{cor:OP2}.
\end{proof}

\begin{remark}[What the construction measures]
The inverse of the automaton $T$ is globally defined and even has a finite formula at each coordinate.  What fails is uniform finite dependence.  A site of rank $n$ may require reading an entire directed chain of length $n$, and the rank track permits such finite chains with no common bound.  Thus non-local finiteness is detected not by an infinite computation in one configuration, but by a family of finite computations of unbounded length.
\end{remark}

\begin{remark}[Relation with earlier infinite-alphabet examples]
For groups containing an infinite-order element, earlier examples use the corresponding infinite cyclic direction to force an inverse with infinite memory; see \cite{CC2011Reversibility}.  The present construction replaces that fixed direction by a direction track taking values in a finite generating set.  The rank track makes every local dependency well-founded, while the existence of arbitrarily long finite pieces of a Cayley-graph ray prevents any uniform inverse memory.  This is why the argument applies equally to finitely generated infinite torsion subgroups.
\end{remark}

\end{document}